\documentclass[a4paper]{article}
\usepackage{hyperref}
\usepackage{amssymb}
\usepackage{amsmath}
\usepackage{amsthm}
\usepackage[utf8]{inputenc}
\usepackage[T1]{fontenc}
\usepackage{mathrsfs}
\usepackage{enumerate}

\newtheorem{theorem}{Theorem}
\newtheorem{corollary}[theorem]{Corollary}
\newtheorem{proposition}[theorem]{Proposition}
\newtheorem{lemma}[theorem]{Lemma}
\theoremstyle{definition}
\newtheorem{definition}[theorem]{Definition}
\theoremstyle{remark}
\newtheorem{remark}[theorem]{Remark}

\newcommand{\ud}{\mathrm{d}}
\newcommand{\pd}{\partial}
\newcommand{\e}{\varepsilon}
\newcommand{\bC}{\mathbb{C}}
\newcommand{\bN}{\mathbb{N}}
\newcommand{\bR}{\mathbb{R}}
\newcommand{\cA}{\mathcal{A}}
\newcommand{\ccD}{\mathcal{D}}
\newcommand{\cE}{\mathcal{E}}
\newcommand{\cG}{\mathcal{G}}
\newcommand{\cN}{\mathcal{N}}
\newcommand{\cT}{\mathcal{T}}
\newcommand{\csub}{\subset \subset}
\newcommand{\coleq}{\mathrel{\mathop:}=}

\newcommand{\fX}{\mathfrak{X}}
\DeclareMathOperator{\Fl}{Fl}
\DeclareMathOperator{\TO}{TO}
\DeclareMathOperator{\id}{id}
\DeclareMathOperator{\supp}{supp}
\DeclareMathOperator{\dist}{dist}
\DeclareMathOperator{\pr}{pr}

\newcommand{\hers}{\hat\cE^r_s}
\newcommand{\hersm}{(\hat\cE^r_s)_m}
\newcommand{\hnrs}{\hat\cN^r_s}
\newcommand{\hgrs}{\hat\cG^r_s}

\newcommand{\hersmM}{(\hat\cE^r_s)_m(M)}
\newcommand{\hersM}{\hat\cE^r_s(M)}
\newcommand{\hgM}{{\hat\cG(M)}}
\newcommand{\heM}{{\hat\cE(M)}}
\newcommand{\hemM}{\hat\cE_m(M)}

\newcommand{\GhrsM}{\hat\cG^r_s(M)}
\newcommand{\NhrsM}{\hat\cN^r_s(M)}
\newcommand{\GhM}{\hat\cG(M)}
\newcommand{\Gd}{\cG^d}
\newcommand{\cTrsM}{\cT^r_s(M)}
\newcommand{\cTrsN}{\cT^r_s(N)}
\newcommand{\cTsrM}{{\ensuremath{\cT^s_r(M)}}}
\newcommand{\DprsM}{\ccD'^r_s(M)}
\newcommand{\DprsN}{\ccD'^r_s(N)}
\newcommand{\DprsU}{\ccD'^r_s(U)}
\newcommand{\cTrsU}{\mathcal{T}^r_s(U)}
\newcommand{\cTsrU}{\mathcal{T}^s_r(U)}
\newcommand{\ocM}{{\Omega^n_c(M)}}
\newcommand{\ocN}{{\Omega^n_c(N)}}
\newcommand{\oc}[1]{{\Omega^n_c(#1)}}

\newcommand{\DpM}{\ccD'(M)}
\newcommand{\DpU}{\ccD'(U)}
\newcommand{\TrsM}{\mathrm{T}^r_s(M)}
\newcommand{\TsrM}{\mathrm{T}^s_r(M)}
\newcommand{\tang}{\mathrm{T}}
\newcommand{\Cinf}{{C^\infty}}

\newcommand{\LX}{\Lie_X}
\newcommand{\Lie}{\mathrm{L}}

\newcommand{\Liep}{\Lie'}
\newcommand{\Lin}{\mathrm{L}}

\newcommand{\iors}{{\iota^r_s}}
\newcommand{\sirs}{\sigma^r_s}

\DeclareMathOperator{\sign}{sign}

\newcommand{\TM}{{\tang M}}

\providecommand{\norm}[1]{\left\lVert#1\right\rVert}
\providecommand{\abso}[1]{\left\lvert#1\right\rvert}

\newcommand{\CdCn}[1]{\mathcal{N}^C(#1)}

\newcommand{\CinfM}{{C^\infty(M)}}

\newcommand{\sD}{\mathscr{D}}

\newcommand{\trans}{\mathrm{T}}
\newcommand{\scale}{\mathrm{S}}

\newcommand{\flowdom}{\sD}

\newcommand{\sk}[2]{\widetilde\cA_{#1}(#2)}
\newcommand{\lsk}[2]{\widetilde\cA_{#1}(#2)}
\newcommand{\ub}[2]{\hat\cA_{#1}(#2)}
\newcommand{\lub}[2]{\cA_{#1}(#2)}

\title{Nonlinear tensor distributions on Riemannian manifolds}
\author{Eduard Nigsch\footnote{University of Vienna, Faculty of Mathematics, Nordbergstr.~15, A-1090 Vienna, Austria, eduard.nigsch@univie.ac.at}}

\begin{document}
\maketitle
\begin{abstract}
We construct an algebra of nonlinear generalized tensor fields on manifolds in the sense of J.-F.~Colombeau, i.e., containing distributional tensor fields as a linear subspace and smooth tensor fields as a faithful subalgebra. The use of a background connection on the manifold allows for a simplified construction based on the existing scalar theory of full diffeomorphism invariant Colombeau algebras on manifolds, still having a canonical embedding of tensor distributions. In the particular case of the Levi-Civita connection on Riemannian manifolds one obtains that this embedding commutes with pullback along homotheties and Lie derivatives along Killing vector fields only.
\end{abstract}

\vskip 1em


\tableofcontents

\section{Introduction} 

\noindent While the theory of distributions developed by S.~L.~Sobolev and L.~Schwartz as a generalization of classical analysis is a powerful tool for many applications, in particular in the field of linear partial differential equations, it is inherently linear and thus not well-suited for nonlinear operations. In particular, one cannot define a reasonable intrinsic multiplication of distributions (\cite{MOBook}). Even more, if one aims at embedding the space $\ccD'(\Omega)$ of distributions on some open set $\Omega \subseteq \bR^n$ into a differential algebra one is limited by the Schwartz impossibility result \cite{Schwartz} which in effect states that there can be no associative commutative algebra $\cA(\Omega)$ satisfying the following conditions:
\begin{enumerate}[(i)]
 \item There is a linear embedding $\ccD'(\Omega) \to \cA(\Omega)$ which maps the constant function $1$ to the identity in $\cA(\Omega)$.
 \item $\cA(\Omega)$ is a differential algebra with linear derivative operators satisfying the Leibniz rule.
 \item The derivations on $\cA(\Omega)$ extend the partial derivatives of $\ccD'(\Omega)$.
 \item The product in $\cA(\Omega)$ restricted to $C^k$-functions for some $k<\infty$ coincides with the usual pointwise product.
\end{enumerate}
However, it was found that such a construction is indeed possible if one requires (iv) for smooth functions only.

In the 1980s J.\ F.\ Colombeau developed a theory of generalized functions (\cite{ColTop,ColNew,ColElem,GKOS,MOBook}) displaying maximal consistency with both the distributional and the smooth theory under the restrictions dictated by the Schwartz impossibility result. A Colombeau algebra thus has come to mean a differential algebra as above containing the space of distributions as a linear subspace and the space of smooth functions as a faithful subalgebra.

The basic idea behind Colombeau algebras is to represent distributions as families of smooth functions obtained through some regularization procedure. The space of these families is then subjected to a quotient construction which ensures that the pointwise product of smooth functions is preserved. In practice one distinguishes two variants of Colombeau algebras, namely the full and the special variant. Full algebras possess a canonical embedding of distributions which allows for a more universal approach to physical models. Special algebras use a fixed mollifier for the embedding and thus are more restrictive but have a considerably simpler structure.

In the context of the special algebra on manifolds (\cite{AB,RD}) the development of generalized counterparts of elements of classical semi-Riemannian geometry was comparatively easy, leading to concepts like generalized sections of vector bundles (thus generalized tensor fields), point values, Lie and covariant derivatives, generalized vector bundle homomorphisms etc.\ (\cite{genpseudo,intrinsic,genconcurv}). However, the embedding into the special algebra is not only non-canonical, it is essentially non-geometric
(\cite[Section 3.2.2]{GKOS}). Therefore the construction of a full variant was desired.

After several attempts and preliminary work by various authors (\cite{colmani,Jelinek,sotonTF}) the full diffeomorphism invariant algebra $\Gd(\Omega)$ of generalized functions on open subsets $\Omega \subseteq \bR^n$ was presented in \cite{found}, which in turn led to the introduction of the full algebra $\hgM$ of generalized functions on a manifold $M$ in intrinsic terms in \cite{global}. On key element in this construction are \emph{smoothing kernels}, which allow for coordinate-independent regularization of scalar distributions.

The latest cornerstone in the development of geometric Colombeau algebras outlined here was the construction of a full Colombeau-type algebra of generalized tensor fields on manifolds as in \cite{global2}. Although the resulting space can be seen as a scalar extension of $\cTrsM$, i.e., a space of the form $\GhM \otimes \cTrsM$ one cannot simply use a coordinatewise embedding of $\DprsM \cong \DpM \otimes_{\CinfM} \cTrsM$: for this to be well-defined one would require the embedding $\iota: \ccD'(M) \to \GhM$ to be $\CinfM$-linear, which cannot be the case; we refer to \cite[Section 4]{global2} for an in-depth discussion of the obstructions to such tensorial extensions of generalized function algebras like $\hgM$.

The underlying deeper reason for this is that regularization of distributional tensor fields in a coordinate-invariant way requires some additional structure on the manifold in order to relate to each other the values of a tensor field at different points, namely a connection. Formally, one uses for this the derived concept of \emph{transport operators} (see Section \ref{sec_transpop}).

In the scalar case the choice between a fixed mollifier (i.e., smoothing kernel) for the embedding on the one side or the parametrization of generalized functions by all possible mollifiers on the other side results in the split between the special and the full variants of the theory, trading in simplicity for generality. In the tensor case the same situation is found: one can either fix a particular regularization procedure resp.\ transport operator for the embedding or parametrize generalized tensor fields by all possible transport operators. Hence on a basic level one can distinguish the following four characteristic types of generalized tensor algebras in the sense of Colombeau:

\begin{enumerate}
 \item {\bf Full-full}: Both mollifiers and transport operators appear as parameters of generalized objects. This results in a construction of considerable technical complexity but having very desirable properties: the embedding of tensor distributions commutes with any Lie derivatives and pullbacks along arbitrary diffeomorphmism (\cite{global2} Propositions 6.6 and 6.8). Furthermore, this algebra can be defined on any (oriented) manifold without further structure.
 \item {\bf Full-special}: If a covariant derivative is given -- in particular the Levi-Civita derivative on a Riemannian manifold -- one can use the associated transport operator (Definition \ref{def_transpop}) for the embedding, which leaves only the mollifiers as parameters for generalized tensor fields. This type is examined in this article, building on the scalar case of \cite{global}. Basically, one sees that in the category of Riemannian manifolds one hase nice properties: the embedding commutes with pullback along isometries (even homotheties) and with Lie derivatives along Killing vector fields.
 \item {\bf Special-full}: This would amount to using a fixed regularization procedure for scalars but a parametrization by all possible transport operators for the tensorial part and has not been investigated so far.
 \item {\bf Special-special}: The most simple case, this has been addressed in the literature before as cited above. Here one uses a non-canonical coordinate-dependent embedding which has obvious drawbacks but gives a simple workable theory.
\end{enumerate}

It should be noted that intermediate variants are conceivable; by treating the special-full case this work finishes one major case and is is intended to stand as a precursor to an all-encompassing study of the above types in one framework.

To recapitulate, in this work we will assume that a fixed covariant derivative is given on the manifold. This allows us to carry out a construction of a space of generalized tensor fields similar to \cite{global2}, but instead of introducing an additional parameter for the generalized objects we use the covariant derivative for embedding distributional tensor fields.

In Section \ref{sec_preliminaries} we will introduce some notation and basic definitions. In Section \ref{sec_algconst} the space of generalized tensor fields on a manifold is constructed. In Section \ref{sec_transpop} we describe transport operators; these are the formal objects used for smooting tensor distributions. In Section \ref{sec_pullback} we treat pullback and Lie derivative of generalized tensor fields. In Section \ref{sec_embed} we give the definition of the embedding of distributional tensor fields, using the background connection in an essential way. In Section \ref{sec_commut} we finally study  commutation relations of pullback along diffeomorphisms and Lie derivatives with the embedding of tensor distributions. The main result is that these commute for homotheties resp.\ Killing vector fields, but not in general.

\section{Preliminaries}\label{sec_preliminaries}

We write $A \csub B$ when $A$ is a compact subset of the interior of $B$.
The identity mapping is denoted by $\id$.
We will frequently use $I=(0,1]$ as an index set.
The topological boundary of a set $U$ is denoted by $\pd U$.
For notions of algebra we refer to \cite{Bourbaki}.
For any open set $V \subseteq \bR^n$, $\oc V$ denotes the space of compactly supported $n$-forms on $V$.

The space of smooth mappings between any subsets $U$ and $V$ of finite-dimensional vector spaces (or manifolds) is $\Cinf(U,V)$, we write $\Cinf(U)$ if $V=\bR$ or $\bC$. We use the usual Landau notation $f(\e) = O(g(\e))$ ($\e \to 0$) if there exist positive constants $C$ and $\e_0$ such that $\abso{f(\e)} \le C g(\e)$ for all $\e \le \e_0$. $\ccD(\Omega)$ denotes the space of test functions on an open subset $\Omega \subseteq \bR^n$ and $\ccD'(\Omega)$ its dual. We use multi-index notation for partial derivatives.

For differentiation theory on infinite-dimensional locally convex spaces we refer to \cite{KM} for a complete exposition of calculus on convenient vector spaces as we use it and to \cite{global2} for background information more specific to our setting. The differential $\ud\colon \Cinf(U, F) \to \Cinf(U, L(E,F))$ is that of \cite[Theorem 3.18]{KM}. Several smoothness arguments are identical to the corresponding ones in \cite{global2} and will only be referred to at the appropriate places.

Our basic references for differential geometry are \cite{Klingenberg,Marsden}. A manifold will always mean a second countable 
Hausdorff smooth manifold of finite dimension that is (except in Section \ref{sec_transpop}) oriented. This dimension will be denoted by $n$ throughout if not otherwise stated.
Charts are written as a pair $(U, \varphi)$ with $U$ an open subset of the manifold and $\varphi$ a homeomorphism from $U$ to an open subset of $\bR^n$.
A vector bundle $E$ with base $M$ is denoted by $E \to M$, its fiber over the point $p \in M$ by $E_p$. The space of sections of $E$ is denoted by $\Gamma(E)$, the space of sections with compact support by $\Gamma_c(E)$, and the space of sections with compact support in a set $L \subseteq M$ by $\Gamma_{c,L}(E)$.
$\tang M$ resp.\ $\tang^* M$ is the tangent resp.\ cotangent bundle of $M$, $\Lambda^n\tang^*M$ is the vector bundle of exterior $n$-forms on $M$.
A particular vector bundle we will use is $\Gamma(\pr_2^*\TrsM)$, the pullback of the tensor bundle $\TrsM$ along the projection of $M \times M$ onto the second factor.
$\fX(M)$ resp.\ $\fX^*(M)$ is the space of vector resp.\ covector fields, $\ocM$ denotes the space of compactly supported $n$-forms and $\cTrsM$ the space of $(r,s)$-tensor fields on $M$. We set $\cT^0_0(M) \coleq \CinfM$. $\ccD(M)$ is the space of test functions on $M$, i.e., the space of smooth functions with compact support. For a diffeomorphism $\mu\colon M \to N$ between manifolds $M$ and $N$, $\mu^*$ denotes pullback of whatever object in question along $\mu$ and we set $\mu_* \coleq (\mu^{-1})^*$. $\tang\mu$ is the tangent map of $\mu$, $(\tang\mu)^r_s$ the corresponding map on the tensor bundle $\TrsM$. The result of the action of a tensor field $t \in \cTrsM$ on a dual tensor field $u \in \cTsrM$ is written as $t \cdot u$. $\LX$ denotes the Lie derivative with respect to a vector field $X$, the flow of $X$ is written as $(t,p) \mapsto \Fl^X_t p$ with $t \in \bR$ and $p \in M$. Given a covariant derivative on $M$ a subset of $M$ is called (geodesically) convex if any two of its points can be connected by a unique geodesic contained in this set.

If $M$ is endowed with a Riemannian metric $g$ we speak of the Riemannian manifold $(M,g)$. The norm induced by $g$ is denoted by $\norm{\cdot}_g$, a metric ball of radius $r>0$ about $p \in M$ with respect to $g$ by $B^g_r(p)$. Following the notation of
\cite[Definition 1.5.1]{Klingenberg}
a covariant derivation is a mapping $\fX(M) \times \fX(M) \to \fX(M)$ locally determined by a family of Christoffel symbols, which are smooth mappings $\Gamma\colon \varphi(U) \to \Lin^2(\bR^n \times \bR^n, \bR^n)$ for each chart $(U, \varphi)$ satisfying the appropriate transformation rule.

Our basic reference for distributions on manifolds is \cite[Section 3.1]{GKOS}.
We define the spaces of scalar distributions resp.\ of $(r,s)$-tensor distributions on an oriented manifold $M$ as
\[ \DpM \coleq \ocM'\quad\text{resp.}\quad \DprsM \coleq \Gamma_c(\tang^s_rM \otimes \Lambda^n\tang^*M)' \]
where the spaces of compactly supported sections carry the usual (LF)-topology and $\DpM$ and $\DprsM$ are endowed with the strong dual topology. We will furthermore make use of the isomorphic representations (\cite[Corollary 3.1.15]{GKOS})
\begin{align*}
 \DprsM \cong \Lin_{\CinfM}(\cTsrM, \DpM) \cong \DpM \otimes_\CinfM \cTrsM.
\end{align*}
The action of a tensor distribution $T \in \DprsM$ will accordingly be denoted by either of
$\langle T, \xi \rangle = \langle T, s \otimes \omega \rangle = \langle T(s), \omega \rangle$ with $\xi$ corresponding to $s \otimes \omega$ via the isomorphism
$ \Gamma_c(\TsrM \otimes \Lambda^n\tang^*M) \cong \cTsrM \otimes_\CinfM \ocM. $
By $\cE'(\Omega) \subseteq \ccD'(\Omega)$ we denote the space of distributions with compact support in $\Omega \subseteq \bR^n$ (this is only used in Section \ref{sec_commut}).

Given a chart $(U, \varphi)$ on $M$, to each distribution $T \in \DpU$ there corresponds a unique distribution in $\ccD'(\varphi(U))$ also denoted by $T$ such that for all $\omega \in \oc U$ with support in $U$ and local representation $\omega(x) = f(x)\, \ud x^1 \wedge \dotsc \wedge \ud x^n$ with $f \in \ccD(\varphi(U))$ the relation $\langle T, \omega \rangle = \langle T, f \rangle$ holds. More explicitly we may also write $\langle T(p), \omega(p) \rangle = \langle T(x), f(x) \rangle$.

For $T \in \DprsU$ and $s \otimes \omega \in \cTsrU \otimes_{\Cinf(U)} \oc U$ we write
$\langle T, s \otimes \omega \rangle = \sum \langle T^\lambda, s_\lambda \cdot \omega \rangle$
where the $T^\lambda \in \DpM$ are the coordinates of $T$ and the $s_\lambda \in C^\infty(U)$ are the coordinates of $s$ on $U$ for $\lambda$ in some index set.

Concerning the theory of local diffeomorphism-invariant Colombeau algebras and the corresponding global construction on manifolds we refer to \cite{found,global}.

\section{Generalized tensor fields}\label{sec_algconst}

In this section we will detail the construction of an algebra of generalized tensor fields. As in other variants of Colombeau algebras the basic idea is that generalized objects are families of their smooth counterparts indexed by some parameters which are required for regularizing the corresponding distributional objects. Our case is a direct extension of the full algebra $\GhM$ of \cite{global} and contains it as the special case $r=s=0$. Scalar distributions are regularized by means of compactly supported $n$-forms having integral $1$, the space of which is denoted by $\ub0M$; as seen in Section \ref{sec_embed}, a connection on the tangent bundle provides the further means to also regularize tensor distributions, hence the indexing set for the basic space remains the same: instead of $\Cinf(\ub0M, \CinfM)$ from the scalar theory we simply take $\Cinf(\ub0M, \cTrsM)$. Replacing the absolute value of scalars by the norm of tensors with respect to any Riemannian metric (all of which are equivalent on compact sets) we will be able to use the same notion of moderateness resp.\ negligibility as in the scalar case (cf.~\cite[Definitions 3.10 and 3.11]{global}).

The definitions of this section are given for arbitrary oriented manifolds; for embedding distributions we will later on assume that a covariant derivative is given.

\subsection{Smoothing kernels}

Smoothing kernels lie at the basis of the construction of full diffeomorphism invariant Colombeau algebras. We recall from \cite[Section 3]{global} that for $\Phi \in \Cinf(I \times M, \ocM)$ the Lie derivatives in both slots are given by $\LX\Phi \coleq \LX \circ \Phi$ and $(\LX'\Phi)(\e,p) \coleq \left.\frac{\ud}{\ud t}\right|_{t=0} \Phi(\e, \Fl^X_t p)$ for $\e \in I$ and $p \in M$.

\begin{definition}Let $M$ be an oriented manifold.
A map $\Phi \in \Cinf(I \times M, \ub0M)$ is called a \emph{smoothing kernel} if it satisfies the following conditions for any Riemannian metric $g$ on $M$:\begin{enumerate}
\item[(i)] $\forall K \csub M$ $\exists \e_0,C>0$ $\forall p \in K$ $\forall \e\le\e_0$: $\supp \Phi(\e,p) \subseteq B^g_{\e C}(p)$,
\item[(ii)] $\forall K \csub M$ $\forall l,m \in \bN_0$ $\forall \theta_1,\dotsc,\theta_m,\zeta_1,\dotsc,\zeta_l \in \fX(M)$ we have
\[ \!\!\!\!\!\!\!\!\!\!\sup_{\substack{p \in K\\q \in M}}\norm{(\Lie_{\theta_1}\dotsc \Lie_{\theta_m} (\Liep_{\zeta_1} + \Lie_{\zeta_1}) \dotsc (\Liep_{\zeta_l} + \Lie_{\zeta_l})\Phi)(\e,p)(q)}_g = O(\e^{-n-m}). \]
\end{enumerate}
The space of all smoothing kernels is denoted by $\sk 0M$. For each $k \in \bN$ we denote by $\sk kM$ the set of all $\Phi \in \sk 0M$ such that $\forall f \in \CinfM$ the value of 
$\abso{f(p) - \int_M f \cdot \Phi(\e,p)}$ is $O(\e^{k+1})$ uniformly on compact sets.
\end{definition}

By the following Lemma whose straightforward proof is omitted (cf.~\cite[Lemma 3.4]{global}) this definition is independent of the metric used.

\begin{lemma}\label{lama}Let $(M, g)$ and $(N, h)$ be Riemannian manifolds. Given a diffeomorphism $\mu\colon M \to N$ and a compact set $K \csub M$ there exists a constant $C>0$ such that
 \begin{enumerate}
  \item[(i)] $\norm{(\mu^*t)(p)}_g \le C \norm{t(\mu(p))}_h$ $\forall t \in \cTrsN$ $\forall p \in K$.
  \item[(ii)] $\norm{(\mu^*\omega)(p)}_g \le C \norm{\omega(\mu(p))}_h$ $\forall \omega \in \ocN$ $\forall p \in K$.
  \item[(iii)] $B_r^g(p) \subseteq \mu^{-1}(B_{r C}^h(\mu(p))) = B_{rC}^{\mu^*h}(p)$ for all small $r>0$ and $\forall p \in K$.
 \end{enumerate}
\end{lemma}

One furthermore easily verifies the following.

\begin{proposition}\label{lahmah}Let $M,N$ be oriented manifolds and $\mu\colon M \to N$ an orientation preserving diffeomorphism. Then the map $\mu^*\Phi\colon I \times M \to \ocM$ defined by $(\mu^*\Phi)(\e,p) \coleq \mu^*(\Phi(\e,\mu(p)))$ is in $\sk kM$.
\end{proposition}

We call $\mu^*\Phi$ the pullback of $\Phi$. Due to the canonical isomorphism of $\oc \varphi(U)$ and $\Cinf(\varphi(U))$ a smoothing kernel $\Phi \in \ub0U$ on a chart $(U, \varphi)$ has \emph{local expression} $\phi \in \Cinf(I \times \varphi(U), \Cinf(\varphi(U)))$ (for details see \cite{gfproc}).

\subsection{The basic spaces}

We now introduce the basic space and appropriate moderateness and negligibility tests. Let $r,s \in \bN_0$ throughout.

\begin{definition}\label{basedef}The \emph{basic space} of generalized $(r,s)$-tensor fields on an oriented manifold $M$ is defined as $\hers(M) \coleq \Cinf(\ub0M, \cTrsM)$.
An element $R \in \hers(M)$ is called \emph{moderate} if it satisfies, for any Riemannian metric $g$ on $M$,
\begin{equation*}
\begin{split}
\forall K \csub M\ \forall l \in \bN_0\ \exists N \in \bN\ \forall X_1,\ldots,X_l \in \fX(M)\\
\forall\ \Phi \in \sk 0M: \sup_{p \in K}\norm{\Lie_{X_1} \ldots \Lie_{X_l} R(\Phi(\e,p))(p)}_g = O(\e^{-N})
\end{split}
\end{equation*}
and \emph{negligible} if additionally it satisfies
\begin{equation*}
\begin{split}
\forall K \csub M\ \forall l,m \in \bN_0\ \exists k \in \bN\ \forall X_1,\ldots,X_l \in \fX(M)\\
\forall \Phi \in \sk kM: \sup_{p \in K}\norm{\Lie_{X_1} \ldots \Lie_{X_l} R(\Phi(\e,p))(p)}_g = O(\e^m).
\end{split}
\end{equation*}
where for $r=s=0$ the norm is replaced by the absolute value. The spaces of moderate and negligible generalized $(r,s)$-tensor fields on $M$ are denoted by $\hersm(M)$ and $\hnrs(M)$, respectively.
\end{definition}

By Lemma \ref{lama} this definition is independent of the metric used. The following Lemma will alleviate the need to use further constructions with cut-off functions in several proofs.

\begin{lemma}\label{tricklein}In Definition \ref{basedef} one can replace ``$\forall \Phi \in \sk 0M$'' resp.\ ``$\forall \Phi \in \sk kM$'' by ``$\exists U \supseteq K$ open $\forall \Phi \in \sk 0U$'' resp.\ ``$\exists U \supseteq K$ open $\forall \Phi \in \sk kU$''. Furthermore, one can instead of ``$\exists U \supseteq K$ open'' demand ``$\forall U \supseteq K$ open''. Finally, in testing one can assume that $K$ is contained in a chart domain.
\end{lemma}
\begin{proof}The only nontrivial part is to show that ``$\exists U \supseteq K$ open'' implies ``$\forall U \supseteq K$ open''. Let $U,V \subseteq M$ both be open subsets of $M$ and let $R \in \hers(M)$ satisfy the moderateness resp.\ negligibility test for all $\Phi \in \sk kU$. Let $K \csub U \cap V$. Given $\Psi \in \sk kV$, let $0 < \delta < \dist(K, \pd (U \cap V))$. Choose $\theta \in \ccD(M)$ with $\supp \theta \subseteq B_\delta(K)$ and $\theta=1$ on $\overline{B}_{\delta/2}(K)$. Let $\e_0>0$ such that $\supp \Psi(\e,p) \subseteq B_\delta(K)$ for all $\e < \e_0$ and $p \in \supp \theta$. With $\lambda \in \Cinf(\bR)$ such that $\lambda=1$ on $(-\infty, \e_0/2]$ and $\lambda=1$ on $[\e_0, \infty)$ define $\Phi \in \sk kU$ by
 \[ \Phi(\e,p) \coleq (1-\lambda(\e)\theta(p)) \Psi_0(\e,p) + \lambda(\e)\theta(p)\Psi(\e,p) \]
where $\Phi_0 \in \sk kU$ is arbitrary. Then for $\e \le \e_0/2$ and $p \in B_{\delta/2}(K)$, $R(\Psi(\e,p))$ equals $R(\Phi(\e,p))$ which satisfies the respective test.
The last claim is clear.
\end{proof}

$\hers(M)$, $\hersm(M)$, and $\hnrs(M)$ are $\CinfM$-modules and $\NhrsM$ is a submodule of $\hersmM$ so we can form the quotient space.
\begin{definition}The \emph{space of generalized $(r,s)$-tensor fields} is defined as the quotient $\CinfM$-module $\hgrs(M) \coleq \hersm(M)/\hnrs(M)$. Smooth tensor fields are embedded into $\hersM$ via the $\CinfM$-linear mapping
$\sirs\colon \cTrsM \to \hers(M)$, $\sirs(t)(\omega) \coleq t$.
\end{definition}

Evidently $\sirs$ has moderate values. The corresponding mapping into the quotient $\GhrsM$ is easily seen to be injective.

For $r=s=0$ the above definitions reproduce -- up to an application of the exponential law $\Cinf(\ub0M, \CinfM) \cong \Cinf(\ub0M \times M)$ -- exactly the global algebra $\hgM$ and the related spaces $\heM$, $\hemM$, and $\hat\cN(M)$ of \cite{global} as well as the embedding $\sigma: \CinfM \to \GhM$. By standard algebraic techniques one obtains the following useful isomorphisms.

\begin{proposition}\label{isos}For $(r,s) \ne 0$ one has the canonical isomorphisms
\[
\hersM \cong \Lin_\CinfM(\cTsrM, \heM) \cong \heM \otimes_{\CinfM} \cTrsM
\]
explicitly given by $R(u)(\omega) \coleq R(\omega) \cdot u$, $(F\otimes t)(\omega) \coleq  F(\omega) \cdot t$, and $(F \otimes t)(u) \coleq F \cdot \sigma(t \cdot u)$ for $R \in \hersM$, $F \in \heM$, $t \in \cTrsM$, and $u \in \cTsrM$.
\end{proposition}
These restrict to isomorphisms on the appropriate subspaces of moderate and negligible functions and thus induce similar isomorphisms for $\GhrsM$. In the sequel we will use the notation $R(u)$ as above or equally $R \cdot u$ without furter notice.

Proposition \ref{isos} also says that $\hersM$ is obtained from $\cTrsM$ by extending its ring of scalars. As a consequence the tensor product of $R = F \otimes t$ and $R' = F' \otimes t'$ (we use the equality sign for elements corresponding via the isomorphism $\hersM \cong \heM \otimes_\CinfM \cTrsM$) is given by $R \otimes R' = (F \cdot F') \otimes (t \otimes t')$. One defines the tensor algebra in the usual way.

\section{Transport operators}\label{sec_transpop}

As will be seen in Section \ref{sec_embed}, in order to regularize distributional tensor fields in a coordinate-invariant way one needs a connection on the tangent bundle for relating fibers over different points of the manifold. In order to formalize the regularization process so-called \emph{transport operators}, which provide linear mappings between any two fibers of the tensor bundle, were introduced in \cite{sotonTF} und further developed in \cite{global2}. We will detail their construction and their relation to covariant derivatives here. Given a covariant derivative on a manifold there is a natural way to obtain a transport operator: locally (in convex neighborhoods) any two points are connected by a unique geodesic along which we can parallel transport tensor fields.

We introduce the following definitions. Let $M,N$ be arbitrary manifolds (not necessarily orientable). For any two vector bundles $E \to M$ and $F \to N$ we define the vector bundle
\[ \TO(E,F) \coleq \bigcup_{(p,q) \in M \times N} \{ (p,q)\} \times \Lin(E_p, F_q). \]
The fiber over $(p,q)$ consists of the space of linear maps from $E_p$ to $F_q$. A section of $\TO(E,F)$, called \emph{transport operator}, is locally given by a smoothly parametrized matrix.

We will now define a transport operator coming from any covariant derivative $\nabla$ on $M$. Let $(U, \varphi)$ be a chart on $M$ and set $U' \coleq \varphi(U)$. Let $\Gamma$ denote the Christoffel symbol of $\nabla$ on $U$. From standard results of ODE theory (cf. \cite{Amann}) the geodesic equation
\begin{equation}\label{geodODE}
\dot u=v,\ \dot v=-\Gamma(u)(v,v),\ u(0) = x \in U',\ v(0) = w \in \bR^n
\end{equation}
has a unique solution $(u,v)(t,x,w)$ defined for $t$ in an open interval $J(x,w) \subseteq \bR$ containing $0$. $(u,v)$ is defined and smooth on the open set
\[ \{\, (t,x,w)\ |\ x \in U', w \in \bR^n, t \in J(x,w)\,\} \subseteq \bR \times U' \times \bR^n.\]
By differentiating system \eqref{geodODE} one obtains the following. Note that differentiation with respect to time $t$ will be denoted by a overhead dot as in $\dot \sigma(t,x,y)$ while a prime denotes the differential with respect to all space variables, as in $\sigma'(t,x,y)\cdot (\xi_1, \xi_2)$ or $X'(x) \cdot \xi_1$.
\begin{lemma}\label{uvder}For $x \in U'$ and $t \in J(x, 0)$ the mappings $u$, $v$, and their derivatives in directions $(\xi_1, \eta_1),(\xi_2, \eta_2) \in \bR^n \times \bR^n$ are given by
\begin{gather*}
u(t,x,0) = x, \quad v(t,x,0) = 0, \\
u'(t,x,0) \cdot (\xi_1, \eta_1) = \xi_1 + t\eta_1, \quad v'(t,x,0) \cdot (\xi_1, \eta_1) = \eta_1, \\
u''(t,x,0) \cdot ((\xi_1, \eta_1),(\xi_2, \eta_2)) = -t^2/2\cdot (\Gamma(x) (\eta_1, \eta_2) + \Gamma(x) (\eta_2, \eta_1)), \\
v''(t,x,0) \cdot ((\xi_1, \eta_1),(\xi_2, \eta_2)) = -t\cdot (\Gamma(x)(\eta_1, \eta_2) + \Gamma(x)(\eta_2, \eta_1)).
\end{gather*}
\end{lemma}

Now fix $p \in U$ and set $x_0 \coleq \varphi(p)$. By continuity of $(u,v)$ we can find for all $r_1>0$ with $r_1 < \dist(x_0, \pd U')$ and $r_2 >0$ real numbers $r_3,r_4,r_5>0$ with $r_4<\dist(x_0,\pd U')$ such that \[ (u,v)(B_{r_3}(0) \times B_{r_4}(x_0) \times B_{r_5}(0)) \subseteq B_{r_1}(x_0) \times B_{r_2}(0) \subseteq U' \times \bR^n. \] Noting for all $a \ne 0$ and $t \in a^{-1} J(x,w)$ the identities
\[
 u(at, x, w) = u(t,x,aw)\textrm{ and } av(at, x, w) = v(t,x,aw)
\]
we see that $(u,v)(B_2(0) \times B_{r_4}(x_0) \times B_{r_3r_5/2}(0)) \subseteq B_{r_1}(x_0) \times B_{2r_2/r_3}(0)$ which means that these geodesics are defined for $\abso{t}<2$. Define now the mapping
\begin{align*}
F: B_{r_4}(x_0) \times B_{r_3r_5/2}(0) & \to B_{r_4}(x_0) \times B_{r_1}(x_0) \\
(x,w) &\mapsto (x, u(1,x,w))
\end{align*}
As seen from Lemma \ref{uvder} $F$ has regular derivative at $(x_0, 0)$ thus there is an open neighborhood $W \subseteq B_{r_4}(x_0) \times B_{r_3r_5/2}(0)$ of $(x_0,0)$ which is mapped diffeomorphically onto an open neighborhood $W_1 \subseteq B_{r_4}(x_0) \times B_{r_1}(x_0)$ of $(x_0,x_0)$. Choose an open neighborhood $W_2$ of $x_0$ with $W_2 \times W_2 \subseteq W_1$ and set  $U_1 \coleq F^{-1}(W_2 \times W_2)$. We have a diffeomorphism $F|_{U_1}: U_1 \to W_2 \times W_2$, which means that any two points $x,y \in W_2$ can be connected by a geodesic $\sigma(t,x,y) \coleq u(t,F^{-1}(x,y))$ which is unique in $B_{r_1}(x_0)$. The set $W_2$ can be chosen such that this geodesic is contained and unique in $W_2$ for all $x,y \in W_2$, i.e., $\varphi^{-1}(W_2)$ is convex; we will assume this to be the case without proof (which can be found for example in \cite[Chapter I Theorem 6.2]{Helgason}). Furthermore we remark that $W_2$ can be chosen arbitrarily small. We denote the initial direction of the geodesic from $x$ to $y$ by $w(x,y) \coleq (\pr_2 \circ (F|_{U_1})^{-1})(x,y)$ where $\pr_2$ is the projection on the second factor.
Parallel transport of a vector $\zeta \in \bR^n$ along $\sigma$ now is defined as the solution of the ODE system
\begin{equation}\label{ptODE}
\rho(0,x,y,\zeta) = \zeta,\quad \dot \rho(t,x,y,\zeta) = -\Gamma(\sigma(t,x,y))(\dot \sigma(t,x,y),\rho(t,x,y,\zeta))
 \end{equation}
which exists for all $t$ for which $\sigma$ is defined; this is linear in $\zeta$. We finally define the prospective transport operator locally as
\begin{align}\label{transpopdef}
a\in \Cinf(W_2 \times W_2, \Lin(\bR^n, \bR^n)),\quad a(x,y) \cdot \zeta \coleq \rho(1,x,y,\zeta)
\end{align}
which determines an element of $\Gamma(W_2, \TO(\TM, \TM))$ on the manifold.
Set $W_p \coleq \varphi^{-1}(W_2)$; performing the above construction for all $p \in M$ one obtains a transport operator on the set $W' \coleq \bigcup_{p \in M} (W_p \times W_p)$ which is an open neighborhood of the diagonal in $M \times M$. This transport operator is denoted by $A_{W'} \in \Gamma(W', \TO(\TM, TM))$ and locally given by $a$. Now the diagonal in $M \times M$ has a closed neighborhood $V$ contained in $W'$ and there is a smooth bump function $\chi \in \Cinf(M \times M)$ with $\chi=1$ on $V$ and $\supp \chi \subseteq W'$ which permits us to extend $A_{W'}$ to a globally defined transport operator $A \in \Gamma(M, \TO(\TM, \TM))$ given by $A(p,q)\coleq \chi(p,q) A_{W'}(p,q)$.

At this point a remark concerning the uniqueness and well-definedness of the construction is in order. Suppose on two charts we constructed local transport operators $a, \tilde a$ whose domains of definition on the manifold overlap. Then on the intersection of this domain they have to agree because of the local uniqueness of geodesics and because parallel transport is independent of the chart, thus $A_{W'}$ is well-defined.

Although $A$ depends on $V$ and $\chi$ it is unique near the diagonal in the sense that each point $p \in M$ has an open neighborhood $U$ (which can be chosen arbitrarily small) such that any two points in $U$ can be joined by a unique geodesic in this set and $A$ is given by parallel transport along these geodesics on $U$. We call $A$ associated to $\nabla$ as in the following definition.

\begin{definition}\label{def_transpop}A transport operator $A \in \Gamma(M, \TO(\TM, \TM))$ is said to be \emph{associated to a covariant derivative} $\nabla$ on $M$ if $A$ is given locally by parallel transport along geodesics with respect to $\nabla$, as in \eqref{transpopdef}.
\end{definition}

Acting on vectors $A$ is denoted by $A^1_0$. $A$ acts on covectors by the adjoint of its inverse: given a covector $\omega_p \in \tang_p^*(M)$ and a vector $v_q \in \tang_qM$ we set $((A^0_1(p,q)\omega_p) \cdot v_q \coleq \omega_q \cdot (A^1_0(p,q)v_p)$. $A$ extends in the usual way to a transport operator on the tensor bundle, i.e., for all $(r,s)$ we have $A^r_s \in \Gamma(\TO(\TrsM, \TrsM))$ given by $A^r_s = (A^1_0)^{\otimes r} \otimes (A^0_1)^{\otimes s}$. By setting $A^0_0 = \id$ we can say that $A$ commutes with tensor products, i.e., $A^{r+p}_{s+q}(s \otimes t) = A^r_s(s) \otimes A^p_q(t)$ for any tensors $s \in \cTrsM$, $t \in \cT^p_q(M)$ with $r,s,p,q \in \bN_0$. If clear from the context we omit the rank and simply write $A$ instead of $A^r_s$.

We will now calculate the derivatives of $A$ explicitly in a chart as we will need them later.
From \eqref{transpopdef} we see that derivatives of $a$ are obtained as derivatives of $\rho$. 
For the direction of differentiation we use arbitrary vectors $e = (\xi_1, \eta_1)$ and $f = (\xi_2, \eta_2) \in \bR^n \times \bR^n$. Then $\sigma$ and its derivatives are given by
\begin{equation}\label{sigmader}
 \begin{split}
\sigma(t,x,y) &= u(t,x,w(x,y)) \\
\sigma'(t,x,y)\cdot e &= u'(t,x,w(x,y) \cdot (\xi_1, w'(x,y) \cdot e) \\
\sigma''(t,x,y)\cdot(e,f) &= u'(t,x,w(x,y)) \cdot (0, w''(x,y) \cdot (e,f))\\
+ u''(t,x&,w(x,y)) \cdot ((\xi_1, w'(x,y) \cdot e), (\xi_2, w'(x,y) \cdot f ))
\end{split}
\end{equation}
and similarly for $\dot\sigma$ with $v$ in place of $u$. The mapping $w$ is given by the second component of the inverse of $G \coleq F|_{U_1}$ which is defined on $W_2 \times W_2$.
The derivative of $G^{-1}$ at $(x,x)$ is given by
\[ (G^{-1})'(x,x) = (G'(x,0))^{-1} = \left( \begin{array}{cc}
\id & 0 \\
-\id & \id \end{array} \right)
\]
and thus $w'(x,x)(\xi, \eta) = \eta - \xi$. Applying the chain rule to $(G^{-1} \circ G)''(x,w)=0$
results by Lemma \ref{uvder} in
\[ w''(x,x)\cdot ((\xi_1, \eta_1), (\xi_2, \eta_2)) = 1/2 \cdot (\Gamma(x)(\eta_1 - \xi_1, \eta_2 - \xi_2) + \Gamma(x)(\eta_2 - \xi_2, \eta_1 - \xi_1)). \]
Inserting this into \eqref{sigmader} we obtain the derivatives of $\sigma$:
\begin{gather*}
\sigma(t,x,x) = x,\quad \sigma'(t,x,x)(\xi, \eta) = \xi + t(\eta- \xi) \\
\sigma''(t,x,x)((\xi_1, \eta_1),(\xi_2, \eta_2)) = (t - t^2)/2 \cdot (\Gamma(x)(\eta_1 - \xi_1, \eta_2 - \xi_2) +\\ 
\Gamma(x)(\eta_2 - \xi_2, \eta_1 - \xi_1)) 
\end{gather*}
Now the derivatives of $\rho$ can be obtained by differentiating \eqref{ptODE}, giving first (omitting the arguments of $\rho$)
\begin{align*}
\rho'(0) \cdot e &= 0 \\
\dot \rho' \cdot e &= -(\Gamma'(\sigma) \cdot \sigma' \cdot e)(\dot \sigma, \rho) - \Gamma(\sigma) (\dot\sigma' \cdot e, \rho) - \Gamma(\sigma)(\dot\sigma, \rho' \cdot e)
 \\
\intertext{and then}
\rho''(0) \cdot (e,f) &= 0, \\
\dot\rho''\cdot (e,f) &= -(\Gamma''(\sigma)(\sigma' \cdot e, \sigma' \cdot f))(\dot \sigma, \rho) - (\Gamma'(\sigma) \cdot \sigma'' \cdot (e,f))(\dot \sigma, \rho ) \\
&\quad - (\Gamma'(\sigma) \cdot \sigma' \cdot e)(\dot\sigma' \cdot f, \rho) - (\Gamma'(\sigma) \cdot \sigma' \cdot e)(\dot\sigma, \rho' \cdot f) \\
&\quad - (\Gamma'(\sigma) \cdot \sigma' \cdot f)(\dot\sigma' \cdot e, \rho) - \Gamma(\sigma)(\dot\sigma'' \cdot (e,f), \rho) \\
&\quad - \Gamma(\sigma)(\dot\sigma' \cdot e, \rho' \cdot f) - (\Gamma'(\sigma) \cdot \sigma' \cdot f)(\dot\sigma, \rho' \cdot e) \\
&\quad - \Gamma(\sigma)(\dot\sigma' \cdot f, \rho' \cdot e) - \Gamma(\sigma)(\dot\sigma, \rho'' \cdot (e,f)).
\end{align*}

Solving these equations we finally obtain
\begin{align*}
\rho(t,x,x,\zeta) &= \zeta, \\
\rho'(t,x,x,\zeta)(\xi, \eta) &= - t \cdot \Gamma(x)(\eta-\xi, \zeta) \\
\rho''(t,x,x,\zeta) \cdot ((\xi_1, \eta_1),(\xi_2, \eta_2)) &= -(\Gamma'(x) \cdot (t\xi_1 + t^2(\eta_1 - \xi_1)/2))(\eta_2 - \xi_2, \zeta) \\
&\quad - (\Gamma'(x) \cdot (t \xi_2 + t^2 (\eta_2 - \xi_2)/2))(\eta_1 - \xi_1, \zeta) \\
&\quad- (t-t^2)/2 \cdot \Gamma(x)(\Gamma(x)(\eta_1 - \xi_1, \eta_2 - \xi_2),\zeta) \\
&\quad- (t-t^2)/2 \cdot \Gamma(x)(\Gamma(x)(\eta_2 - \xi_2, \eta_1 - \xi_1), \zeta) \\
&\quad + t^2/2 \cdot \Gamma(x)(\eta_1 - \xi_1, \Gamma(x)(\eta_2 - \xi_2, \zeta )) \\
&\quad + t^2/2 \cdot \Gamma(x)(\eta_2 - \xi_2, \Gamma(x)(\eta_1 - \xi_1, \zeta)).
\end{align*}

This holds for all $\abso{t}<2$ and $x,y$ in the open neighborhood $W_2$ of $x_0$. As $x_0$ was arbitrary in $U'$ we have shown the following.

\begin{lemma}\label{transpop_abl}Let $(U, \varphi)$ be an arbitrary chart on $M$. Then the local representation $a \in \Cinf(\varphi(U) \times \varphi(U), \Lin(\bR^n, \bR^n))$ of a transport operator $A$ associated to $\nabla$ satisfies the following identities for all $x \in \varphi(U)$ and $\xi, \eta, \zeta \in \bR^n$:
\begin{align*}
\textrm{(i) }& a(x,x) = \id, \\
\textrm{(ii) }& a'(x,x)(\xi, \eta) \cdot \zeta = -\Gamma(x)(\eta-\xi, \zeta), \\
\textrm{(iii) }& 2 a''(x,x)((\xi_1, \eta_1), (\xi_2, \eta_2))\cdot \zeta = -(\Gamma'(x) \cdot (\eta_1 + \xi_1))(\eta_2 - \xi_2, \zeta) \\
&\quad - (\Gamma'(x) \cdot (\eta_2 + \xi_2))(\eta_1 - \xi_1, \zeta) + \Gamma(x)(\eta_1 - \xi_1, \Gamma(x)(\eta_2 - \xi_2, \zeta)) \\
&\quad + \Gamma(x)(\eta_2 - \xi_2, \Gamma(x)(\eta_1 - \xi_1, \zeta)).
\end{align*}
\end{lemma}

Finally we recall that the pullback $(\mu, \nu)^*A \in \Gamma(\TO(\tang N, \tang N))$ of a transport operator $A \in \Gamma(\TO(\TM, \TM))$ along a pair of diffeomorphisms $\mu, \nu: N \to M$ is given by
\begin{equation*}
((\mu, \nu)^* A)(p,q) \coleq (\tang_q\nu)^{-1} \cdot A(\mu(p), \nu(q)) \cdot \tang_p \mu
\end{equation*}
and its Lie derivative $\Lie_{X \times Y}A \in \Gamma(\TO(\TM, \TM))$ along a pair of vector fields $X,Y \in \fX(M)$ by
\begin{equation}\label{aprilbeta}
(\Lie_{X \times Y} A)(p,q) \coleq \left.\frac{\ud}{\ud \tau}\right|_{\tau=0} ((\Fl^X_\tau, \Fl^Y_\tau)^* A)(p,q)
\end{equation}
We abbreviate $\Lie_{X \times X}A$ by $\Lie_XA$. By Lemma \ref{transpop_abl} (i) $\LX A$ can not be associated to any covariant derivative on $M$. See \cite[Appendix A]{global2} for further details about transport operators.
We finally note that trivially, the restriction of a transport operator associated to a covariant derivative is associated to the restriction of the covariant derivate.

\section{Pullback and Lie derivatives}\label{sec_pullback}

In this section we will define pullback along a diffeomorphism and Lie derivatives of generalized tensor fields. Let $M,N$ be oriented manifolds.

\begin{definition}\label{defpb}Let $\mu\colon M \to N$ be an orientation preserving diffeomorphism and $R \in \hers(N)$. The map $\mu^*R \in \hersM$ defined by $(\mu^*R)(\omega) \coleq \mu^* ( R ( \mu_* \omega))$ for $\omega \in \ub0M$ is called the \emph{pullback} of $R$ along $\mu$.
\end{definition}

\begin{remark}Essentially, this is the only sensible definition in our context; in fact, assuming that $\mu^*: \hers(N) \to \hersM$ commutes with contractions one can by Proposition \ref{isos} contract with dual smooth tensor fields which reduces everything to the choice of the pullback of scalar fields. The latter we naturally (or for consistency with the scalar case of \cite{global}) assume to be given by $\mu^*: \hat\cE^0_0(N) \to \hat\cE^0_0(M)$, $(\mu^*F)(\omega) \coleq F(\mu_*\omega) \circ \mu$.
\end{remark}

\begin{lemma}The map $\mu^*\colon \hers(N) \to \hers(M)$ of Definition \ref{defpb} preserves moderateness and negligibility and thus defines a map $\mu^*\colon \hgrs(N) \to \hgrs(M)$.
\end{lemma}
\begin{proof}
Given $R \in \hers(N)$ and $\Phi \in \sk kM$, by Definition \ref{basedef} moderateness and negligibility of $\mu^*R$ are established by evaluating Lie derivatives of the tensor field $t \in \cTrsM$ defined by
\[ t(p) \coleq (\mu^*R)(\Phi(\e,p))(p) = \mu^*(R(\mu_*(\Phi(\e,p))))(p) \]
on a compact set $K \csub M$. Given an arbitrary vector field $X \in \fX(M)$, $\Lie_Xt$ is $\mu^*(\Lie_{\mu_*X}\mu_*t)$, where $(\mu_*t)(p) = R(\mu_*(\Phi(\e,\mu^{-1}(p))))(p) = R((\mu_*\Phi)(\e,p))(p)$.
By Proposition \ref{lahmah} $\mu_*\Phi$ is in $\sk kN$, thus the growth conditions on $\Lie_Xt$ (and similarly for any number of Lie derivatives) are obtained directly from those of $R$ with help of Lemma \ref{lama} (ii).
\end{proof}

Given $R \in \hers(M)$ we can define its Lie derivative $\LX R \in \hers(M)$ along a complete vector field $X \in \fX(M)$ in a geometric manner via its flow, namely as $(\LX R)(\omega) \coleq \frac{\ud}{\ud t}|_{t=0} ((\Fl^X_t)^*R)(\omega)$ for $\omega \in \ub0M$. By the chain rule this is seen to be equal to $-\ud R(\omega)(\Lie_X\omega) + \Lie_X(R(\omega))$ (see \cite[Section 6]{global2} for the smoothness argument). Thus the Lie derivative is formally the same as for elements of $\hgM$ (\cite[Definition 3.8]{global}).
For non-complete vector fields we use this formula for defining the Lie derivative.

\begin{definition}For $X \in \fX(M)$ we define the Lie derivative $\LX R$ of $R \in \hersM$ as $(\LX R)(\omega) \coleq -\ud R(\omega)(\Lie_X\omega) + \Lie_X(R(\omega))$.
\end{definition}

\begin{lemma}\label{prodrule}The Lie derivative $\LX: \hersM \to \hersM$ commutes with the tensor product and with contractions.
\end{lemma}
\begin{proof}
Let $R = F \otimes t \in \hersM \cong \heM \otimes \cTrsM$ and $R' = F' \otimes t'$. A direct calculation shows that $\LX R = \LX F \otimes t + F \otimes \LX t$ and consequently $\LX(R \otimes R') = \LX R \otimes R' + R \otimes \LX R'$. Contracting with a smooth dual tensor field $v$ the general case follows from
\begin{align*}
 \LX(R \cdot v)(\omega) &= \LX((R \cdot v)(\omega)) - \ud (R \cdot v)(\omega)(\LX \omega) \\
&= \LX(R(\omega) \cdot v) - \ud R(\omega)(\LX \omega) \cdot v \\
&= \LX (R(\omega)) \cdot v + R(\omega) \cdot \LX v - \ud R(\omega)(\LX \omega) \cdot v \\
&= ((\LX R)\cdot v + R \cdot \LX v)(\omega).
\end{align*}
\end{proof}
\begin{corollary}$\LX\colon \hers(M) \to \hers(M)$ preserves moderateness and negligibility.
\end{corollary}
\begin{proof}By Proposition \ref{isos} we know that $R \in \hers(M)$ is moderate resp.~negligible if and only if $R \cdot t$ is moderate resp.~negligible for all $t \in \cTsrM$. By Lemma \ref{prodrule} $(\LX R)\cdot t = \LX(R \cdot t) - R \cdot \Lie_Xt$, so the claim follows because $\LX\colon \heM \to \heM$ preserves moderateness and negligibility (\cite[Theorem 4.6]{global}).
\end{proof}

It follows that $\LX$ is also defined on $\GhrsM$.

\section{Embedding of distributional tensor fields}\label{sec_embed}

Using a transport operator we can approximate a locally integrable $(r,s)$-tensor field $t$ at a point $p \in M$ by $t(p) \sim \int A^r_s(q,p)t(q)\omega(q)\,\ud q$, where $\omega \in \ub0M$ has support in a small ball around $p$. This approximation is valid if $A(q,q)$ is the identity for all $q$ in a neighborhood of $p \in M$, see Proposition \ref{regapprox} below for a precise statement. In order to obtain a distributional formula which we can use for the embedding we examine the action of $t$ on a dual tensor field $u$ of rank $(s,r)$:
\begin{align*}
t(p) \cdot u(p) & \sim \int ( A^r_s(q,p)t(q) \cdot u(p) ) \omega(q)\,\ud q\\
& = \int ( t(q) \cdot A^s_r(p,q)u(p) ) \omega(q)\,\ud q\\
& = \langle t(q), A^s_r(p, q)u(p) \otimes \omega(q) \rangle.
\end{align*}

These considerations lead to the following definition of an embedding of $\DprsM$ into $\GhrsM$.

\begin{definition}Let $M$ be an oriented manifold with covariant derivative $\nabla$ and $A \in \Gamma(\TO(\TM, TM))$ a transport operator associated to $\nabla$. Then we define an embedding $\iors\colon \DprsM \to \hersM$ by setting
\[ ((\iors)(\omega) \cdot v)(p) \coleq \langle t, A(p, \cdot)v(p) \otimes \omega \rangle \]
where $t \in \DprsM$, $\omega \in \ub0M$, $v \in \cTsrM$, and $p \in M$.
\end{definition}

If the rank is clear from the context we simply write $\iota$ instead of $\iors$; this way we may index the embedding by the covariant derivative used, as in $\iota_\nabla$. It is easily seen that $\iota$ is a presheaf morphism, i.e., commutes with restriction (one can show that $\GhrsM$ is a fine sheaf). As seen from the next Lemma in the case $\mu=\id$ two transport operators associated to the same covariant derivative give the same embedding into the quotient.

We recall that given a diffeomorphism $\mu: M \to N$ one can define the pullback of a covariant derivative $\nabla$ on $N$ by $(\mu^*\nabla)_X Y \coleq \mu^*(\nabla_{\mu_*X}\mu_*Y)$ for $X,Y \in \fX(M)$, which gives a covariant derivative $\mu^*\nabla$ on $M$.

\begin{lemma}\label{isocomm}Let $\mu: M \to N$ be an orientation preserving diffeomorphism and suppose there are covariant derivatives $\tilde\nabla$ on $M$ and $\nabla$ on $N$. Let the embedding $\iota$ on $M$ resp. $N$ use any transport operator associated to $\tilde \nabla$ resp. $\nabla$.
Then $\iota_{\tilde \nabla} \circ \mu^* - \mu^* \circ \iota_\nabla$ has values in $\NhrsM$ if and only if $\tilde \nabla = \mu^* \nabla$.
\end{lemma}
\begin{proof}We first assume that $\tilde \nabla = \mu^*\nabla$. Fix $K \csub M$ for testing. We may assume that $K$ is contained in an open convex set $U_0$. Let $L$ be a compact neighborhood of $K$ in $U_0$. Given $\Phi \in \sk 0M$, there exists $\e_0>0$ such that $\Phi(\e,p)$ has support in $U_0$ for all $\e<\e_0$ and $p \in L$. Now let $\tilde A$ and $A$ denote any transport operators associated to $\tilde\nabla$ and $\nabla$, respectively. We then claim that for all $p \in L$, $t \in \DprsN$, $v \in \cTsrM$, and $\omega \in \ub0M$ with support in $U_0$ the expression
\begin{align*}
(\iota_{\tilde \nabla} ( \mu^* t)(\omega) \cdot v)(p) &= \langle \mu^*t, \tilde A(p, \cdot) v(p) \otimes \omega \rangle  = \langle t, \mu_*(\tilde A(p, \cdot) v(p)) \otimes \mu_* \omega \rangle\\
\intertext{equals}
(\mu^* ( \iota_\nabla t)(\omega) \cdot v)(p) &= (\mu^*((\iota t)(\mu_* \omega)) \cdot v)(p) = \mu^* ( (\iota t)(\mu_* \omega) \cdot \mu_* v)(p) \\
& = ((\iota t) (\mu_* \omega) \cdot \mu_* v)(\mu (p)) = \langle t, A(\mu (p), \cdot) \mu_* v (\mu (p)) \otimes \mu_* \omega \rangle.
\end{align*}
These expressions are equal if
$\mu_*(\tilde A(p, \cdot)v(p))(\mu (q)) = A(\mu (p), \mu (q)) (\mu_* v)(\mu (p))$
 for $q \in U_0$. But this is clear in the case $\tilde \nabla = \mu^*\nabla$ because then $\mu$ preserves geodesics, convex sets, and parallel displacement.

For the converse, by writing $\iota_{\tilde \nabla} \circ \mu^* - \mu^* \circ \iota_\nabla = (\iota_{\tilde \nabla} - \iota_{\mu^*\nabla} ) \circ \mu^* + \iota_{\mu^*\nabla}\circ \mu^* - \mu^* \circ \iota_{\nabla}$ and because $\mu^*: \hers(N) \to \hers(M)$ is bijective we only have to show that $\iota_{\tilde \nabla} - \iota_{\mu^*\nabla} \subseteq \hnrs(M)$ implies $\tilde \nabla = \mu^*\nabla$, which will be accomplished by Theorem \ref{nogo} below.
\end{proof}

Because homotheties preserve Levi-Civita connections (\cite[Chapter 3]{ONeill}) we immediately obtain the following (cf.~\cite[Propositions 6.6 and 6.8]{global2}).

\begin{corollary}\label{isokill}
 If $\mu$ is a homothety between Riemannian manifolds then $\iors \circ \mu^* = \mu^* \circ \iors$ has negligible values, where the embeddings use transport operators associated to the Levi-Civita derivatives. Consequently, $\iors \circ \LX = \LX \circ \iors$ for all Killing vector fields $X$.
\end{corollary}

\begin{remark}The (non-trivial) proof that $\iors(t)$ is smooth is to a large extent identical to the corresponding result in \cite[Section 7]{global2}, the necessary modifications being straightforward (we simply have one slot less to deal with).
\end{remark}

We will now show that the embedding $\iors$ has the properties required for an embedding of distributions into Colombeau algebras, namely it has moderate values, for smooth tensor fields it reproduces $\sirs$, and it is injective.

\begin{proposition}The embeddings have the following properties.
\begin{enumerate}[(i)]
\item $\iors(\DprsM) \subseteq \hersmM$.
\item $(\iors-\sirs)(\cTrsM) \subseteq \NhrsM$.
\item For $v \in \DprsM$, $\iors(v) \in \NhrsM$ implies $v=0$.
\end{enumerate}
\end{proposition}

\begin{proof}
(i) For testing we fix $K \csub M$ and $l \in \bN_0$. For any vector fields $X_1,\dotsc,X_l \in \fX(M)$ and a smoothing kernel $\Phi \in \sk 0M$ by Proposition \ref{isos} we need to calculate $\Lie_{X_1} \dotsc \Lie_{X_l}(p \mapsto \langle t, A^s_r(p, \cdot)u(p) \otimes \Phi(\e,p)\rangle )$ on $K$ for arbitrary $u \in \cTrsM$.  By the chain rule (for a detailed argument on why $t$ commutes with the Lie derivative see the proof of \cite[Proposition 6.8]{global2}) this is given by terms of the form
\begin{equation}\label{modeins}
\langle t, v(p, \cdot)\otimes \Liep_{Y_1} \dotsc \Liep_{Y_k} \Phi(\e,p) \rangle
\end{equation}
for some $Y_i \in \fX(M)$ ($i=1\dotsc k \in \bN$) and $v \in \Gamma(\pr_2^*(\TsrM))$; the latter consists of Lie derivatives of $u$ transported by Lie derivatives of $A$. By the definition of smoothing kernels, for $\e$ small enough and $p$ in a relatively compact neighborhood of $K$ the support of $\Phi(\e,p)$ for $p \in K$ lies in a (bigger) relatively compact neighborhood $L$ of $K$. Because $t$ is continuous and linear and $\cTsrM \otimes \ocM$ carries the usual inductive limit topology (as in \cite[Section 2]{global2}), the modulus of \eqref{modeins} can be estimated by a finite sum of seminorms of $\Gamma_{c,L}(\TsrM \otimes \Lambda^n\tang^*M)$ applied to the argument of $t$ in \eqref{modeins}. These seminorms are given by $s \mapsto \sup_{x \in L}\norm{\Lie_{Z_1}\dotsc \Lie_{Z_p}s(x)}$ for some vector fields $Z_j \in \fX(M)$, $j=1,\dotsc,p \in \bN$ (the norm is with respect to any Riemannian metric on $M$). It thus remains to estimate $\norm{\Lie_{Z_1} \dotsc \Lie_{Z_p}(v(p,\cdot) \otimes \Liep_{Y_1} \dotsc \Liep_{Y_k}\Phi(\e,p))}$. This in turn reduces to an estimate of $\Lie$- and $\Liep$-derivatives of $\Phi$, which immediately gives the desired moderateness estimate by definition of the space of smoothing kernels.

(ii) In order to show the claim we have to verify (using Proposition \ref{isos}) that for arbitrary $u \in \cTsrM$, $K \csub M$ and $m \in \bN_0$ there is some $k \in \bN$ such that for all $\Phi \in \sk kM$ we have the estimate
\begin{equation}\label{einss}
\sup_{p \in K}\abso{\int_M (t \cdot (A^s_r(p, \cdot)u(p)))(q) \Phi(\e, p)(q)\,\ud q - (t\cdot u)(p)} = O(\e^m).
\end{equation}
By Lemma \ref{tricklein} we may assume that $K$ is contained in the domain of a chart $(U, \varphi)$ and  $\Phi \in \sk kU$. Defining $f \in \Cinf(U \times U)$ by $f(p,q)\coleq t(q) \cdot A^s_r(p, q)u(p)$ we can write \eqref{einss} as
$\sup_{p \in K}\abso{\int_U \bigl(f(p,q) - f(p,p)\bigr)\Phi(\e,p)(q)\,\ud q}$.
Setting $\tilde f \coleq f \circ (\varphi^{-1} \times \varphi^{-1})$ and $x \coleq \varphi(p)$ the integral is given by 
$\int_{\varphi(U)} (\tilde f(x,y) - \tilde f(x,x))\tilde\phi(\e,x)(y)\,\ud y$
where $\tilde \phi \in \lsk k{\varphi(U)}$ is the local expression of $\Phi$. It is easily verified that this is $O(\e^{k+1})$ uniformly for $x \in \varphi(K)$, so for $k+1\ge m$ the required estimates are satisfied.

(iii) is shown in Corollary \ref{injektivitaet} below.

\end{proof}

Although we will not treat association in full detail the following is a first step in this direction (cf.~\cite[Section 9]{global2} for the type of results that can be obtained). Let
\[
\rho\colon \cTrsM \to \DprsM,\quad \rho(t)(u \otimes \omega) \coleq \int (t \cdot u)\,\omega
\]
be the embedding of $\cTrsM$ into $\DprsM$. Given a tensor distribution $T$ in $\DprsM$ and a smoothing kernel $\Phi \in \sk 0M$ we set $T_\e \coleq [p \mapsto (\iors T)(\Phi(\e,p))(p)]$ which is an element of $\cTrsM$.
$T_\e$ can be seen as a regularization of $T$ which gets more accurate for smaller $\e$. More precisely, we will now show that $\rho(T_\e)$ converges to $T$ weakly in $\DprsM$ for $\e \to 0$.

Fix $u \otimes \omega \in \cTsrM \otimes_\CinfM \ocM$. We may assume that $\omega$ (and thus $u$) has support in a fixed compact set $K$ contained in a chart $(U, \varphi)$: using partitions of unity we can write $u \otimes \omega = \sum_i \chi_i u \otimes \chi_i \omega$ where the $\chi_i$ are smooth functions on $M$ with $\supp \chi_i \subseteq U_i$. Then
$\langle \rho(T_\e) - T, u \otimes \omega \rangle = \sum_i \langle \rho(T_\e) - T, \chi_i u \otimes \chi_i \omega \rangle$
converges to $0$ if the result holds for the case where $K$ is contained in a chart $(U, \varphi)$.

We abbreviate $\tilde u^{j_1 \dotsc j_s}_{i_1 \dotsc i_r}(p,q) \coleq (A^s_r(p,q)u(p))^{j_1\dotsc j_s}_{i_1 \dotsc i_r}$
and note that $u^{j_1\dotsc j_s}_{i_1 \dotsc i_r}(p) = \tilde u^{j_1 \dotsc j_s}_{i_1 \dotsc i_r}(p,p)$.
Given any neighborhood $L$ of $K$ which is relatively compact in $U$ there is as in the proof of Lemma \ref{tricklein} some $\e_0>0$ and a smoothing kernel $\Phi_1 \in \sk 0U$ such that for all $p \in L$ and $\e < \e_0$ the support of $\Phi(\e,p)$ is contained in $U$ and $\Phi(\e,p)|_U = \Phi_1(\e,p)$. Let $\Phi_1$ have local expression $\tilde \phi$. Let $\psi \in \ccD(\varphi(U))$ be determined by $\varphi_*\omega = \psi\, \ud x^1 \wedge \dotsc \wedge \ud x^n$. Then for $\e < \e_0$ (denoting the local expressions of $T^{i_1 \dotsc i_r}_{j_1 \dotsc j_s}$ and $\tilde u^{i_1 \dotsc i_s}_{j_1 \dotsc j_r}$ by the same letter, respectively)
\begin{align*}
 \langle \rho(T_\e), u \otimes \omega \rangle & = \int_M \langle T(q), A^s_r(p,q)u(p) \otimes \Phi(\e,p)(q)\rangle\, \omega(p) \\
&= \int_M \langle T^{i_1 \dotsc i_r}_{j_1 \dotsc j_s}(q), (A^s_r(p,q)u(p))^{j_1\dotsc j_s}_{i_1 \dotsc i_r} \cdot \Phi_1(\e,p)(q) \rangle\,\omega(p) \\
&= \int_{\varphi(U)} \langle T^{i_1 \dotsc i_r}_{j_1 \dotsc j_s}(y), \tilde u^{j_1\dotsc j_s}_{i_1 \dotsc i_r}(x,y)  \cdot \tilde\phi(\e,x)(y) \rangle \psi(x)\, \ud^n x\\
&= \int_{\varphi(U)} \langle T^{i_1 \dotsc i_r}_{j_1 \dotsc j_s}(y), \tilde u^{j_1\dotsc j_s}_{i_1 \dotsc i_r}(x,y)  \cdot \psi(x) \cdot  \tilde\phi(\e,x)(y)\, \ud^n x\\
&= \langle T^{i_1 \dotsc i_r}_{j_1 \dotsc j_s}(y), \int_{\varphi(U)} \tilde u^{j_1\dotsc j_s}_{i_1 \dotsc i_r}(x,y)  \cdot \psi(x) \cdot  \tilde\phi(\e,x)(y)\,\ud^n x \rangle \\
\intertext{and}
\langle T, u \otimes \omega \rangle &= \langle T^{i_1 \dotsc i_r}_{j_1 \dotsc j_s}(p) , u^{j_1 \dotsc j_s}_{i_1 \dotsc i_r}(p) \cdot \omega(p) \rangle = \langle T^{i_1 \dotsc i_r}_{j_1 \dotsc j_s}(y) , u^{j_1 \dotsc j_s}_{i_1 \dotsc i_r}(y) \cdot \psi(y) \rangle.
\end{align*}
Integration here commutes with the distributional action, as can be seen from writing the above as the tensor product of the distribution $T^{i_1 \dotsc i_r}_{j_1 \dotsc j_s}$ with the distribution $1$.
Now for each choice of $j_1,\dotsc,j_s,i_1,\dotsc,i_r$ we abbreviate $f(x,y) \coleq \tilde u^{j_1\dotsc j_s}_{i_1 \dotsc i_r}(x,y)  \cdot \psi(x)$ and note that $f(y,y) = u^{j_1\dotsc j_s}_{i_1\dotsc i_r}(y) \cdot \psi(y)$.
Because the function mapping $y$ to $\int_{\varphi(U)} f(x,y)\tilde \phi(\e,x)(y)\,\ud x - f(y,y)$ has support in a compact set in $\varphi(U)$, for each component of $T_\e - T$ by \cite[Proposition 21.1]{Treves} there exist $m>0$ and $C>0$ such that
\[ \langle (T_\e - T)^{i_1\dotsc i_r}_{j_1\dotsc j_s}, u^{j_1\dotsc j_s}_{i_1\dotsc i_r} \cdot \omega \rangle \le \sup_{\substack{\abso{\alpha}\le m \\ y \in \varphi(U)}} \norm{\pd^\alpha(\int_{\varphi(U)} \!\!\!\!\!\!\!f(x,y) \tilde \phi(\e,x)(y)\, \ud x - f(y,y))} \]
which is $O(\e)$ by \cite[Corollary 5.3]{gfproc} or the proof of \cite[Proposition 9.10]{global2}.
Summarizing, we have shown:

\begin{proposition}\label{regapprox}Given $T \in \DprsM$ and $\Phi \in \sk0M$ the regular distribution given by $p \mapsto (\iors T)(\Phi(\e,p))(p)$
converges weakly to $T$ in $\DprsM$ for $\e \to 0$.
\end{proposition}

\begin{corollary}\label{injektivitaet}For $T \in \DprsM$, $\iors(T) \in \NhrsM$ implies $T=0$.
\end{corollary}
\begin{proof}For suitable $k \in \bN$, $u \otimes \omega \in \cTsrM \otimes_\CinfM \ocM$, and $\Phi \in \sk kM$
\begin{align*}
\abso{\langle T, u \otimes \omega \rangle} & = \abso{\lim_{\e \to 0} \langle (\iors T)(\Phi(\e,p))(p), (u \otimes \omega)(p) \rangle} \\
&= \abso{\lim_{\e \to 0} \int_M \langle T(q), A^s_r(p,q)u(p) \otimes \Phi(\e,p)(q) \rangle\, \omega (p)} \\
&\le \lim_{\e \to 0} \sup_{p \in \supp \omega} \abso{ \langle T(q), A^s_r(p,q)u(p) \otimes \Phi(\e,p)(q) \rangle } \cdot \abso{\int_M \!\!\omega(p)}
\end{align*}
which is $O(\e^m)$ because of negligibility of $T$.
\end{proof}

\section{Commutation relations}\label{sec_commut}

\begin{proposition}The operations $\mu^*$ and $\LX$ on $\hersM$ extend the usual pullback and Lie derivative of smooth tensor fields: $\mu^* \circ \sirs = \sirs \circ \mu^*$ and $\LX \circ \sirs = \sirs \circ \Lie_X$.
\end{proposition}
\begin{proof}For $t \in \cTrsN$ and $\omega \in \ub0M$ we have
\begin{align*}
\mu^*(\sirs(t))(\omega) &= \mu^*(\sirs(t)(\mu_*\omega)) = \mu^* t = \sirs(\mu^* t)(\omega) \\
\intertext{and for $t \in \cTrsM$, $X \in \fX(M)$, and $\omega \in \ub0M$}
\Lie_X(\sirs(t))(\omega) &= -\ud (\sirs(t))(\omega)(\Lie_X\omega) + \Lie_X(\sirs(t)(\omega)) = \Lie_Xt \\
&= \sirs(\Lie_X t)(\omega).\qedhere
\end{align*}
\end{proof}

In Corollary \ref{isokill} we already saw that the embedding of distributional tensor fields commutes with pullback along homotheties and consequently with Lie derivatives along Killing vector fields.

Lemma \ref{isocomm} allows to reformulate the question of whether pullback along an arbitrary (orientation preserving) diffeomorphism $\mu\colon M \to N$ commutes with $\iors$, for if one endows $M$ with the pullback metric $\mu^*h$ this question reduces to checking whether the embeddings $(\iota^g)^r_s$ and $(\iota^{\mu^*h})^r_s$ arising from the Riemannian metrics $g$ and $\mu^*h$ are equal. We then have the following main result.

\begin{theorem}\label{nogo}Let $\nabla$ and $\tilde\nabla$ be covariant derivatives on $M$ with corresponding embeddings $\iota$ and $\tilde \iota$, respectively. Then
\begin{enumerate}
\item[(i)] $(\iota^r_s - \tilde\iota^r_s)(\DprsM) \subseteq \NhrsM$ implies $\nabla = \tilde\nabla$.
\item[(ii)] $\iota$ does not commute with arbitrary Lie derivatives.
\end{enumerate}
\end{theorem}

The proof consists of several steps. First, the assumptions are written as conditions having the same form, namely negligibility of the generalized function $(\omega, p) \mapsto \langle T, Z(p,\cdot) \otimes \omega \rangle \in \cE(M)$ for all $T \in \DprsM$ and some $Z \in \Gamma(\pr_2^*(\TsrM))$. Then, choosing $T$ appropriately we obtain that derivatives of $Z$ in the second slot vanish. Finally, the derivatives of $Z$ are calculated explicitly. This involves the derivatives of the transport operator, which are related to the connection as seen in Lemma \ref{transpop_abl}.

Beginning with the first step, we show that both $\iota-\tilde\iota$ and $\iota \circ \Lie_X - \LX \circ \iota$ give rise to expressions of the same form. In the first case, the equality $\iota = \tilde \iota$ in the quotient means that for all $T \in \DprsM$ the generalized function $R \coleq (\iota - \tilde \iota)T \in \hersmM$ given by
\begin{equation}\label{nulleins}
(R(\omega) \cdot v)(p) = \langle T, (A(p,\cdot) - \tilde A(p, \cdot))v(p) \otimes \omega \rangle
\end{equation}
for $v \in \cTsrM$ and $\omega \in \ocM$ is negligible, where $A$ resp.\ $\tilde A$ are transport operators associated to $\nabla$ resp.\ $\tilde \nabla$. Note that the difference $(p,q) \mapsto (A(p, q) - \tilde A(p, q))v(p)$ is an element of $\Gamma(\pr_2^*(\TsrM))$ and vanishes on the diagonal in $M \times M$.

In the second case, from the proof of \cite[Proposition 6.8]{global2} (in particular, equations (6.13) and (6.14) therein) we immediately obtain the identity
\begin{equation}\label{nullzwei}
((\iota \circ \Lie_X - \LX \circ \iota)(T) (\omega) \cdot v)(p) = \langle T, (\Lie_{X \times X}A)(p, \cdot)v(p) \otimes \omega \rangle
\end{equation}
where the term on the right hand side is exactly the additional term of the Lie derivative of generalized tensor fields in \cite{global2} which makes it commute with the embedding already in the basic space there. As in our case pullback of generalized tensor fields cannot act on the transport operator this term does not cancel. Note that also $(p,q) \mapsto (\Lie_{X \times X}A)(p,\cdot)v(p)$ is an element of $\Gamma(\pr_2^*(\TsrM))$ and vanishes on the diagonal.

Thus in both cases (i) and (ii) for each $v \in \cTsrM$ we have found some $Z \in \Gamma(\pr_2^*(\TsrM))$ such that for all $T \in \DprsM$ the generalized function $R \cdot v \in \hemM$ defined by
\begin{equation}\label{erdnuss}
\omega \mapsto [ p \mapsto \langle T, Z(p, \cdot) \otimes \omega \rangle ]
\end{equation}
is negligible (i.e., an element of $\hat\cN(M)$). The next proposition and the subsequent corollary allow us to get information about $Z$ by the right choices of the distribution $T$.

The idea behind the following proof is the following: locally negligibility of \eqref{erdnuss} means that an expression like $\langle T, f(x, \cdot)\tang_x\scale_\e \varphi \rangle$ converges to $0$. As a simple case consider $n=1$, $x=0$ and $f$ depending on the second slot only with $f(0)=0$.  Then $\langle T, f \cdot S_\e\varphi\rangle \to 0$ one the one hand, but on the other hand we can write this as (neglecting the remainder of the Taylor expansion, which vanishes asymptotically):
\[ \langle T(y), (f(0) + f'(0)\cdot y + \dotsc + f^{(k)}(0)\cdot y^k/k!)S_\e\varphi \rangle \to 0 \]
As the support of $S_\e \varphi$ gets arbitrarily small we can only hope to get information about $f$ at $0$. It vanishes there, but we can determine its derivatives by taking for $T$ the principal value of $1/y$: this gives the terms
\[ f(0) \cdot \langle 1/y, S_\e\varphi \rangle,\quad f'(0) \cdot \langle 1, S_\e\varphi\rangle,\quad \dotsc \quad f^{(k)}(0) \langle y^{k-1}/k!, S_\e\varphi\rangle. \]
If $\varphi$ now has vanishing moments of order $k-1$ and is even the only remaining term is $f'(0)$ so we can conclude $f'(0)=0$.

In the general case the proof is more involved. In what follows $\cE'(\Omega) \subseteq \ccD(\Omega)$ is the space of compactly supported distributions on $\Omega$ and $\cE(\Omega)$, $\cE_M(\Omega)$, and $\cN(\Omega)$ are the basic space of $\Gd(\Omega)$ in C-formalism \cite{found} and its subspaces of moderate resp.\ negligible elements.

\begin{proposition}\label{trick1}Let $\Omega \subseteq \bR^n$ be open and $f \in \Cinf(\Omega \times \Omega)$. Then
\begin{enumerate}
\item[(i)]For each $T \in \ccD'(\Omega)$ the mapping in $\cE(\Omega)$ given by
\begin{equation}\label{endeapril2}
(\varphi, x) \mapsto \langle T, f(x, \cdot) \varphi(.-x) \rangle
\end{equation}
is moderate, i.e., an element of $\cE_M(\Omega)$.
\item[(ii)] If for all compactly supported distributions $T \in \cE'(\Omega)$ the mapping \eqref{endeapril2} is in $\cN(\Omega)$ then all first order partial derivatives in the second slot of $f$ vanish on the diagonal, i.e., $\pd_i(y \mapsto f(x,y))|_x=0$ $\forall x \in \Omega$ $\forall i=1\dotsc n$.
\end{enumerate}
\end{proposition}

\begin{proof}
(i) resembles the statement that the embedding of distributions into $\cE(\Omega)$ has moderate values; the proof is virtually the same (see \cite[Theorem 7.4 (i)]{found}), inserting $f(x, \cdot)$ at the appropriate places. This results in an application of the chain rule and the appearance of some extra constants (suprema of derivatives of $f$ on compact sets), but leaves moderateness intact.

 (ii) Let $x$ be an arbitrary  point of $\Omega \subseteq \bR^n$. Choose some $\eta>0$ with $\eta < \dist(x, \pd\Omega)$ and a smooth bump function $\chi \in \ccD(\bR)$ with $\chi = 1$ on $\overline{B_{\eta/2}}(0)$ and $\supp \chi \subseteq B_\eta(0)$.

Consider the distribution $t \mapsto \sign t \cdot \abso{t}^{n-2}$. For $n>1$ this is a locally integrable function, for $n=1$ this means the principal value of $\frac{1}{t}$. This distributions thus is given for all $n \in \bN$ by
\begin{equation}\label{pv}
\langle \sign t \cdot \abso{t}^{n-2}, \omega \rangle = \lim_{\delta \to 0} \int_\delta^\infty t^{n-2} (\omega(t) - \omega(-t))\, \ud t\qquad \forall \omega \in \ccD(\bR).
 \end{equation}

We introduce the distribution
\[ P \coleq \delta \otimes \dotsc \otimes \delta \otimes \chi(t) \sign t \cdot \abso{t}^{n-2} \otimes \delta \otimes \dotsc \otimes \delta \in \ccD'(\bR^n)\]
or more explicitly
\[ \langle P, \omega \rangle = \langle \sign t \cdot \abso{t}^{n-2}, \chi(t) \omega(0,\dotsc,t,\dotsc,0)\rangle\qquad \forall \omega \in \ccD(\bR^n) \]
where $\chi(t) \sign t \cdot \abso{t}^{n-2}$ resp. $t$ appears at the $k$th position for an arbitrary $k \in \{1,\dotsc,n\}$ which shall be fixed from now on.

$u \coleq \trans_xP = P(.-x)$ then is a compactly supported distribution on $\Omega$: because $\supp P \subseteq \{0\} \times \dotsc \times B_\eta(0) \times \dotsc \times \{0\} \subseteq B_\eta(0)$ we have $\supp u \subseteq B_\eta(x) \subseteq \Omega$.

With $K = \{ x \}$ and arbitrary $m \in \bN$, by negligibility of \eqref{endeapril2} there is some $q \in \bN$ (which can be chosen arbitrarily high) such that for any fixed $\varphi \in \lub q{\bR^n}$ we have
\begin{equation} \label{heissermittwoch}
\langle u, f(x, \cdot)\trans_x\scale_\e\varphi \rangle = O(\e^m) \qquad (\e \to 0).
\end{equation}

Choose $\varphi_1 \in \ccD([0, \infty))$ which is constant in a neighborhood of $0$ and satisfies
\[ \int_0^\infty s^{j/n} \varphi_1(s)\, \ud s = \left\{ \begin{aligned}
\frac{n}{\omega_n} &\qquad j= 0\\
0 &\qquad j=1,2,3,\dotsc,q
\end{aligned} \right.
\]
where $\omega_n$ is the area of the $(n-1)$-dimensional sphere in $\bR^n$.
Such a function exists by a straightforward adaption of the proof of \cite[Proposition 1.4.30]{GKOS}, and we set $\varphi \coleq \varphi_1 \circ \norm{\ }^n \in \ccD(\bR^n)$.
Then $\varphi$ is in $\lub q{\bR^n}$, as a simple calculation shows.

Choosing $r>0$ such that $\supp \varphi \subseteq B_r(0)$, let $\e < \eta/(2r)$ from now on, which implies $\supp \trans_x\scale_\e\varphi \subseteq B_{\eta/2}(x) \subseteq \Omega$ and $\supp [t \mapsto \varphi_1(t^n/\e^n)] \subseteq B_{\eta/2}(0)$. By equation \eqref{pv} the expression $\langle u, f(x,\cdot)\trans_x\scale_\e \rangle$ on the left-hand side of \eqref{heissermittwoch} is given by
 \begin{equation}\label{junizwei}
 \langle P, f(x,x+.)S_\e\varphi\rangle = \lim_{\delta \to 0} \int_\delta^{\eta/2} \chi(t) t^{n-2} ( \tilde f(t) - \tilde f(-t)) \e^{-n} \varphi_1((t/\e)^n)\,\ud t
 \end{equation}
and we can write \eqref{junizwei} as the limit for $\delta \to 0$ of
 \begin{multline*}
\int_\delta^{\eta/2} \sum_{l=0}^{q} t^{n-2} \frac{\tilde f^{(l)}(0)}{l!} (t^l - (-t)^l)\e^{-n}\varphi_1(t^n/\e^n)\,\ud t + \\
\int_\delta^{\eta/2} t^{n-2} \int_0^1 \frac{(1-v)^{q}}{q!} \bigl( \tilde f^{(q+1)}(vt) - (-1)^{q+1}\tilde f^{(q+1)}(-vt)\bigr)\,\ud v\, \cdot\\
 \ t^{q+1} \e^{-n} \varphi_1(t^n/\e^n)\, \ud t.
 \end{multline*}
The terms for even $l$ and odd $l\le 3$ 
vanish and the term for $l=1$ gives exactly $2\tilde f'(0)/\omega_n$. Finally, after substituting $t = \e s^{1/n}$ the remainder term is
given by
\begin{multline*}
\frac{\e^{q}}{n} \int_0^{(\eta/(2\e))^n} \int_0^1 \frac{(1-v)^{q}}{q!} \bigl(\tilde f^{(q+1)}(\e v s^{1/n}) - (-1)^{q+1}\tilde f^{(q+1)}(-\e v s^{1/n})\bigr)\, \cdot \\
s^{q/n} \varphi_1(s)\, \ud v\, \ud s
\end{multline*}
and the integral is bounded by a finite constant independently of $\e$. 
Concluding, from Taylor expansion on the one hand and the assumption on the other hand we have
 \begin{align*}
 \langle u, f(x,\cdot)\trans_x\scale_\e\varphi \rangle &= 2 \tilde f'(0) / \omega_n  + O(\e^q) \textrm{ and} \\
 \langle u, f(x,\cdot)\trans_x\scale_\e\varphi \rangle &= O(\e^m).
 \end{align*}
Together, this gives $\tilde f'(0) = O(\e^{\min(q, m)})$ where $m$ and $q$ can be chosen arbitrarily high.
Thus $\tilde f'(0) = D_2f(x,x)\cdot e_k = 0$, which concludes the proof because $x$ and $k$ were arbitrary.
\end{proof}

Now follows the corresponding result on a manifold.

\begin{corollary}\label{obenabl}Let $Z \in \Gamma(\pr_2^*(\TsrM))$ satisfy $Z(p,p) = 0$ $\forall p \in M$. Then
\begin{enumerate}
\item[(i)] For each $T \in \DprsM$ the mapping from $\ub0M \times M$ into $\bR$ defined by
\begin{equation}\label{endeapril}
(\omega, p) \mapsto \langle T, Z(p, \cdot) \otimes \omega \rangle
\end{equation}
is moderate, i.e., an element of $\hemM$.
\item[(ii)] If for \emph{all} $T \in \DprsM$ the mapping \eqref{endeapril} is negligible then $\Lie_Y(Z(p, \cdot))(p)$ vanishes for all $Y \in \fX(M)$ and $p \in M$.
\end{enumerate}
\end{corollary}
\begin{proof}
As in Proposition \ref{trick1}, (i) follows in the same way as moderateness of embedded distributions (see \cite[Section 5]{global}).

(ii) Let $(U, \psi)$ be a chart on $M$ and $\{b_\lambda\}_\lambda$ a basis of $\cTrsU$ with dual basis $\{b^\lambda\}_\lambda$ of $\cTsrU$. Denote the coordinates of $Z$ on $U$ by $Z_\lambda \in \Cinf(U \times U)$, i.e., $Z(p,q) = Z_\lambda(p,q)b^\lambda(q)$ for all $p,q \in U$.

We will show that for any compactly supported distribution $t_U \in \cE'(\psi(U))$ the mapping 
defined by $(\varphi, x) \mapsto \langle t_U, Z_\lambda(\psi^{-1}(x), \cdot)\varphi(.-x) \rangle$
is an element of $\CdCn{\psi(U)}$, i.e., negligible in the local diffeomorphism invariant scalar algebra in the C-setting. For this purpose define $S \in \DprsU \cong \cTrsU \otimes_{\CinfM} \DpU$ by $S \coleq b_\lambda \otimes t$ (where $t \in \DpU$ corresponds to $t_U$ as in Section \ref{sec_preliminaries}), which has compact support and thus a trivial extension to a distributional tensor field $T \in \DprsM$ with $T|_U = S$. By assumption the map $\ub0M \times M \to \bR$ given by $(\omega,p) \mapsto \langle T, Z(p, \cdot) \otimes \omega \rangle$
is negligible, thus also its restriction to $U$ which is the map $\ub0U \times U \to \bR$ given by
$(\omega,p ) \mapsto \langle T, Z(p,\cdot)\otimes \omega \rangle = \langle T|_U, Z(p,\cdot)|_U \otimes \omega \rangle = \langle t, Z_\lambda(p,\cdot)\omega \rangle$. 
This implies that the corresponding map $\lub 0{\psi(U}) \times \psi(U) \to \bR$ given by
\begin{align*}
(\varphi, x) &\mapsto \langle t, Z_\lambda(\psi^{-1}(x), \cdot)\psi^*(\varphi(.-x) \, \ud y^1 \wedge \dotsc \wedge \ud y^n) \rangle \\
&= \langle t, \psi^*(Z_\lambda(\psi^{-1}(x), \psi^{-1}(\cdot))\varphi(.-x) \, \ud y^1 \wedge \dotsc \wedge \ud y^n) \rangle \\
&= \langle t_U, (Z_\lambda \circ (\psi^{-1} \times \psi^{-1}))(x, \cdot) \varphi(.-x) \rangle
\end{align*}
is in $\CdCn{\psi(U)}$ for any choice of $t_U \in \cE'(\psi(U))$. Proposition \ref{trick1} now implies that $\pd_i (y \mapsto Z_\lambda(\psi^{-1}(x), \psi^{-1}(y)))|_x = 0$ for all $x$ in $\psi(U)$ and all $i$. Noting that $Z(p,p)=0$ by assumption, the local formula for $\Lie_Y(Z(p,\cdot))(p)$ evaluates to $0$.\end{proof}

Returning to the proof of Theorem \ref{nogo} and assuming \eqref{nulleins} resp.~\eqref{nullzwei} to be negligible for all choices of $T$, Corollary \ref{obenabl} implies in the case $(r,s)=(0,1)$ for all $X$, $Y$, $Z \in \fX(M)$ and $p \in M$ the identities
\begin{align*}
\text{(i) }\Lie_Y(q \mapsto (A(p,q) - B(p,q))Z(p) & = 0 \quad\textrm{and} \\
\text{(ii) } \Lie_Y(q \mapsto (\Lie_{X \times X}A)(p,q)Z(p))(p) & = 0.
\end{align*}

Given a vector field $X \in \fX(M)$, by its \textit{local flow} on $U$ we mean the map $\alpha\colon \flowdom(X) \to \varphi(U)$ determined by the ODE
\begin{equation}\label{flowODE}
\alpha(0, x) = x,\quad \dot \alpha(t,x)=X(\alpha(t,x))
\end{equation}
where $X \in \Cinf(\varphi(U), \bR^n)$ is the local representation of $X$ on $U$ and $\flowdom(X)$, the maximal domain of definition of $\alpha$, is an open subset of $\bR \times \varphi(U)$. For $p \in U$, its flow along $X$ is given by $\Fl^X_t p = \tang\varphi^{-1} (\alpha(t, \varphi(p)))$ for all $t$ with $(t, \varphi(p)) \in \flowdom(X)$. Furthermore, $\alpha$ is smooth. By differentiating \eqref{flowODE} one sees that for all $(t,x) \in \flowdom(X)$ the local flow $\alpha$ satisfies 
\begin{equation}\label{anfangswerte}
\begin{gathered}
\alpha(0,x) = x,\quad
\alpha'(0,x) = \id\\
\alpha''(0,x) = 0,\quad
\dot\alpha'(0,x)=X'(x).
\end{gathered}
\end{equation}

For (i), $\Lie_X(q \mapsto A(p, q)Z(p))(p)$ is given by the derivative at $t=0$ of
$\tang \Fl^X_{-t} A(p, \Fl^X_t p) Z(p)$.
This means we have to differentiate the local expression $D\alpha(-t, \alpha(t, x)) a(x, \alpha(t,x)) Z(x)$
which results in
\begin{multline*}
-\dot\alpha'(-t, \alpha(t,x)) a(x, \alpha(t,x)) Z(x)\\
+ \alpha''(-t, \alpha(t,x))X(\alpha(t,x))a(x, \alpha(t,x))Z(x)\\
+ \alpha'(-t, \alpha(t,x))a'(x, \alpha(t,x))(0,X(\alpha(t,x)))Z(x)
\end{multline*}
which by \eqref{anfangswerte} and Lemma \ref{transpop_abl} evaluates to
$-X'(x)Z(x) - \Gamma(x)(X(x), Z(x)) = -\nabla_ZX(x)$ at $t=0$.
As we can choose $X$, $Z$, and $x$ freely this immediately implies that both covariant derivatives are equal, which proves Theorem \ref{nogo} (i).

Now to (ii). 
By equation \eqref{aprilbeta} $\Lie_Y(q \mapsto \Lie_{X \times X}A(p,q)Z(p))(p)$ is given by
\begin{equation}\label{alpha3}
\left.\frac{\ud}{\ud s}\right|_{s=0} T\Fl^Y_{-s} \left.\frac{\ud}{\ud t}\right|_{t=0} \tang_{\Fl^X_tq}\Fl^X_{-t} \cdot A(\Fl^X_t p, \Fl^X_t \Fl^Y_s p) \cdot \tang_p \Fl^X_t \cdot V(p).
\end{equation}
We will first calculate the inner expression, which (setting $q \coleq \Fl^Y_s p$) is given by
\begin{equation}\label{alpha2}
\tang_{\Fl^X_tq}\Fl^X_{-t} \cdot A(\Fl^X_t p, \Fl^X_t q) \cdot \tang_p \Fl^X_t \cdot V(p).
\end{equation}
Note that for $p,q \in U$ and the modulus of $s,t$ small enough the flows in \eqref{alpha2} and \eqref{alpha3} stay inside $U$, thus we have for \eqref{alpha2} the local expression
\[
F(t,x,y) \coleq \alpha'(-t, \alpha(t,y)) a(\alpha(t,x), \alpha(t,y))\cdot \alpha'(t,x)Z(x).
\]
Here $\alpha$ (and below $\beta$) denotes the local flow of $X$ (and $Y$, respectively). The derivative w.r.t.\ $t$ of this is
\begin{equation*}
\begin{split}
\dot F&(t,x,y) = \bigl( -\dot\alpha'(-t, \alpha(t,y))a(\alpha(t,x),\alpha(t,y))\alpha'(t,x)\\
& + \alpha''(-t,\alpha(t,y))X(\alpha(t,y))a(\alpha(t,x),\alpha(t,y))\alpha'(t,x) \\
& +\alpha'(-t, \alpha(t,y))a'(\alpha(t,x),\alpha(t,y))\bigl(X(\alpha(t,x)), X(\alpha(t,y))\bigr)\alpha'(t,x)\\
& +\alpha'(-t,\alpha(t,y))a(\alpha(t,x),\alpha(t,y))\dot\alpha'(t,x)\bigr) Z(x).
\end{split}
\end{equation*}
Evaluating at $t=0$ we obtain by \eqref{anfangswerte} that $F'(0,x,y)$ equals
\begin{align*}
\bigl(-X'(y)a(x,y) + a'(x,y)(X(x), X(y)) + a(x,y)X'(x)\bigr)Z(x).
\end{align*}
Note that for $x=y$ this expression vanishes by Lemma \ref{transpop_abl}. Now we set $y=\beta(t,x)$; then \eqref{alpha3} is locally given by the derivative at $s=0$ of
\begin{multline*}
G(s,x) \coleq \beta'(-s, \beta(s,x))\bigl(-X'(\beta(s,x))a(x,\beta(s,x))\\
+ a'(x,\beta(s,x))(X(x), X(\beta(s,x)) + a(x,\beta(s,x))X'(x)\bigr)Z(x).
\end{multline*}
The derivative of $G$ is
\begin{align*}
\dot G(s,x) = & -\dot\beta'(-s,\beta(s,x))F'(0,x,\beta(s,x)) +\\
& + \beta''(-s,\beta(s,x))Y(\beta(s,x))F'(0,x,\beta(s,x)) \\
& + \beta'(-s,\beta(s,x)) \bigl( -X''(\beta(s,x))Y(\beta(s,x))a(x,\beta(s,x)) \\
&\qquad-X'(\beta(s,x))a'(x,\beta(s,x))(0,Y(\beta(s,x)))\\
&\qquad + a''(x,\beta(s,x))\bigl((X(x), X(\beta(s,x))),(0, Y(\beta(s,x)))\bigr)\\
&\qquad + a'(x,\beta(s,x))(0,X'(\beta(s,x))Y(\beta(s,x)))\\
&\qquad+ a'(x,\beta(s,x))(0,Y(\beta(s,x)))X'(x)\bigr) Z(x)
\end{align*}
and at $s=0$ the first two terms vanish, while for the rest we obtain (omitting $x$ notationally)
\[
\dot G(0,x) = (-X''Y - X'a'(0,Y) + a''((X, Y), (0, Y)) + a'(0,X'Y) + a'(0,Y)X')Z
\]
which by Lemma \ref{transpop_abl} equals
\begin{multline}
 -X''YZ + X' \Gamma(Y, Z) - 1/2 \bigl( \Gamma' \cdot (X + Y)(Y, Z) + (\Gamma' \cdot Y)(Y - X, Z) \\
- \Gamma(Y - X, \Gamma(Y, Z)) - \Gamma(Y, \Gamma(Y - X, Z))\bigr) - \Gamma(X'Y, Z) - \Gamma(Y, X'Z) \\
 = -X''YZ + X' \Gamma(Y, Z) - (\Gamma' \cdot Y)(Y, Z) + \Gamma(Y, \Gamma(Y, Z))\\
 - 1/2 \bigl( (\Gamma' \cdot X)(Y, Z) - (\Gamma' \cdot Y)(X, Z)
+ \Gamma(X, \Gamma(Y, Z)) + \Gamma(Y, \Gamma(X, Z))\bigr)\\
 - \Gamma(X'Y, Z) - \Gamma(Y, X'Z). \label{endeee}
\end{multline}
By assumption, this vanishes for all possible choices of $X$, $Y$, $Z$, and $x$. Setting $X=0$ gives $(\Gamma' \cdot Y)(Y, Z) = \Gamma(Y, \Gamma(Y, Z))$
and, applying this formula to $\Gamma' \cdot (X+Y)(X+Y, Z)$ for any $X,Y,Z$ we obtain
\begin{equation*}
(\Gamma' \cdot X)(Y,Z) + (\Gamma' \cdot Y)(X,Z) = \Gamma(X, \Gamma(Y,Z)) + \Gamma(Y, \Gamma(X,Z))
\end{equation*}
and thus, inserting this into \eqref{endeee}
\begin{equation*}
 -X'' Y Z + X' \Gamma(Y,Z) - (\Gamma' \cdot X)(Y,Z) - \Gamma(X'Y, Z) - \Gamma(Y, X'Z) = 0
\end{equation*}
for all choices of $X,Y,Z,x$. In particular, choosing $X$ constant in a neighborhood of $x$ gives $(\Gamma'\cdot X)(Y, Z)=0$, thus $\Gamma'=0$ and we can drop this term. Then, choosing $X$ such that $X'=\id$ around $x$ implies $\Gamma(Y,Z)=0$. It remains that $X''YZ=0$, which clearly cannot hold for arbitrary $X,Y,Z$. This proves the assertion that $\iors$ cannot commute with arbitrary Lie derivatives.

We thus have established Theorem \ref{nogo}.

\subsection*{Acknowledgments}

This research has been supported by START-project Y237 and project P20525 of
the Austrian Science Fund and the Doctoral College 'Differential Geometry and Lie
Groups' of the University of Vienna.

\end{document}